\documentclass[11pt]{amsart}

\usepackage{mathrsfs}

\usepackage{epstopdf}
\usepackage{amsthm}

\usepackage{dsfont}

\usepackage{amsmath}
\usepackage{amssymb}
\usepackage{graphicx}
\usepackage{amsthm}
\usepackage{eucal}
\usepackage{geometry}
\usepackage[usenames, dvipsnames]{color}
\usepackage{enumerate}
\usepackage[T1]{fontenc}
\usepackage{verbatim}

\usepackage[final]{showkeys}

\usepackage{hyperxmp}
\usepackage{hyperref}
\hypersetup{
pdftitle={Statistical properties of coupled expanding maps on a lattice with general infinite range couplings and H\"older densities},
pdfauthor={Chinmaya Gupta, Nicolai Haydn}
}

\newtheorem{thm}{Theorem}[section]
\newtheorem{proposition}[thm]{Proposition}
\newtheorem{lemma}[thm]{Lemma}
\newtheorem{corollary}[thm]{Corollary}
\newtheorem{defn}[thm]{Definition}

\newcommand{\R}{\mathbb{R}}
\newcommand{\Z}{\mathbb{Z}}
\newcommand{\set}[1]{\left\{#1\right\}}
\newcommand{\I}{\Omega}

\newcommand{\C}{\mathcal{C}}

\newcommand{\ii}{i}
\newcommand{\var}{\text{var}}
\newcommand{\LL}{\ensuremath{\mathcal{L}} }

\title[Statistical Properties for CML]{Statistical properties of coupled expanding maps on a lattice with general infinite range couplings and H\"older densities}
\author[C. Gupta]{Chinmaya Gupta}
\address[Chinmaya Gupta]{University of Southern California, 3620 S. Vermont Ave, Los Angeles, CA, 90089}
\email{chinmaya.gupta@usc.edu}
\urladdr{http://www-bcf.usc.edu/~chinmayg}

\author[N. Haydn]{Nicolai Haydn}
\address[Nicolai Haydn]{University of Southern California, 3620 S. Vermont Ave, Los Angeles, CA, 90089}
\email{nhaydn@usc.edu}
\urladdr{}

\thanks{}

\begin{document}

\maketitle

\begin{abstract}
We continue the development of transfer operator techniques for expanding maps on a lattice coupled by general interaction functions. We obtain a spectral gap for an appropriately defined transfer operator, and, as corollaries, the existence of an invariant conformal probability measure for the system, exponential decay of correlations, the central limit theorem and the almost sure invariance principle.
\end{abstract}

\section{Introduction}

For one dimensional dynamical systems, the conditions under which there exists a unique ergodic invariant probability measure, supported on an invariant attractor and governing the dynamics of initial conditions in the basin of the attractor, are well understood.  The complexity of the geometry in higher dimensions makes the problem much harder, however, in certain cases, significant results have been proven. For instance, the uniformly hyperbolic case has been presented in \cite{MR2423393}, the non-uniformly hyperbolic case of the H\'enon map in \cite{MR1087346, MR1218323}, and recently, the $n$ dimensional analogue in \cite{MR2415378}. In all, even though the 1 dimensional case is better understood in general than the $n$ dimensional setting for $n>1$, the theory of SRB measures for finite dimensional dynamical systems is fairly complete (for an excellent, though somewhat dated, overview, see \cite{MR1933431} and the references therein).

However, outside of a few settings, not much is known in the infinite dimensional setup. The primary problem is one of methodology, and the settings in which something can be said about an infinite dimensional dynamical system are those in which the techniques applicable in the finite dimensional case can be extended to infinite dimensions. Coupled maps on a lattice (CML) provide an example of infinite dimensional systems that can be studied by extensions of finite dimensional techniques, and for this reason, they have been extensively studied.

An excellent review of the theory of CML is provided in \cite{MR1464237}. In recalling the background on CML theory, we will be more restrictive; in particular, we will focus only on the development of ideas relevant to the results in this paper. An imprecise definition of a CML is provided in this section to set the context of the results that we mention. A precise formulation follows in later sections. Suppose we have a map $\tau$ on $[0,1)$ and let $\Omega = [0,1)^\mathbb{Z}$. We have an extended map $\bar \tau$ on the space $\Omega$ which can be defined as $(\bar \tau(x))_i = \tau (x_i)$ where $x\in \Omega$ and $x_i\in [0,1).$ Suppose now we have interactions $E$ between the different dynamical systems $(\tau_i, [0,1))$. In the simplest case of $E$ being the nearest neighbor diffusive coupling, we can specify the interactions as, for a given $0<\epsilon \ll 1$,
\[
(E(x))_i = (1 - \epsilon) x_i + \frac{\epsilon}{2} x_{i -1} + \frac{\epsilon}{2} x_{i+1}.
\] $\epsilon$ therefore becomes a parameter that tunes the strength of the interactions between the nodes of the lattice. We note also that this coupling specified has range 1, because only the closest neighbors to each node influence the state of that node.  The system under study is now iterations of $E\circ \bar \tau$.

Bunimovich and Sina\u i \cite{MR967468} considered $\tau$ as expanding maps of the interval with nearest neighbor diffusive coupling with a non-constant diffusion strength, chosen so that the coupling map was onto. They established that if the interval map exhibits sufficient expansion, then there exists a unique invariant measure, mixing in time and space. They construct the invariant measure as the limit of the Gibbs measures for the finite dimensional projections of their coupled system. Then, in \cite{MR1184471} the authors implement a numerical algorithm to extract orbits periodic in time and space for a 1-d lattice of H\'enon maps, coupled by a weak nearest neighbor diffusive coupling. They define a new family of Lyapunov exponents that estimate the growth rate of spatially inhomogeneous perturbations. 

Following the results in \cite{MR967468} and \cite{MR1184471}, a natural question arose: if the coupling is viewed as perturbation of the uncoupled system, what can happen if the coupling strength becomes large relative to the inherent stability of the uncoupled system? This question was studied in \cite{MR1184470}, where the authors considered CML, where the maps on each node were  one dimensional with a globally attracting stable periodic trajectory. These maps are coupled by a diffusive nearest neighbor coupling. The authors prove that if the coupling strength is sufficient large, then the phase space of the coupled map lattice is split into a complicated partition with many basins of different attractors.

The question of phase transitions for a one parameter family of finite-dimensional CML, with the coupling strength $\epsilon$ as the parameter of interest, was studied further, theoretically and numerically, by \cite{MR1192055}. They obtained sufficient conditions in terms of $\epsilon$ that the lattice have a continuum of ergodic components when the individual node maps are the doubling maps. They also produced an example of a function $f$ for which the uncoupled system is mixing, whereas for a suitable $\epsilon$, there are several domains in the phase space, interchanging with the period 2.

Keller and K\"unzle \cite{MR1176625} then investigated transfer operator techniques for CML maps. They also considered expanding maps on each node of the lattice, and coupled them by a weak coupling (small $\epsilon$). The authors established the spectral theory using spaces of bounded variation, closely following the setting in \cite{MR967468}. Then, in a three-part paper, Gundlach and Rand \cite{MR1211566, MR1211567, MR1211568} developed the stable manifold theory for coupled lattice maps with short or finite range interactions (Part I), and use this to establish the existence of a natural spatio-temporal measure that plays the role of the usual SRB measure in the case of temporal systems (Part III) and the existence and uniqueness of Gibbs states for higher dimensional symbolic systems by using thermodynamic formalism (Part II). The transfer operator approach of \cite{MR1176625} was developed further by \cite{MR1981170} where the author established a spectral gap for weakly coupled real-analytic circle maps. In \cite{MR1936014}, the author then modified the approach in \cite{MR1981170} to construct generalized transfer operators associated to potentials  and establish a spectral gap for small potentials for weakly coupled analytic maps. Some limit theorems, such as the central limit theorem, moderate deviations, and a partial large deviations result were also established.

Returning to the setting of expanding real maps, \cite{MR2083421} gave a proof of the existence, uniqueness and exponential mixing of invariant measures for weakly coupled lattice maps without cluster expansion. The coupling considered was finite range and had a very specific form, and the transfer operator was considered on the space of measures of bounded variation with absolutely continuous finite dimensional marginals. Subsequently, efforts were made to admit more general couplings and more general maps on the lattice nodes, and in \cite{MR2395729} the authors study a one-dimensional lattice of weakly coupled piecewise expanding maps of the interval. Strong assumptions are still required on the coupling, however, the authors do not require that the coupling be a homeomorphism of the infinite dimensional state space. They prove that the transfer operator defined on an appropriate space of densities with bounded variations, with absolutely continuous finite dimensional marginals (with respect to the Lebesgue measure), has a spectral gap. This implies that there exists a measure with exponential decay of correlations in time and space. In \cite{MR2200881} the authors extend the results established in \cite{MR2395729} to include lattices of any dimensions and couplings of infinite rage with the coupling strength decaying exponentially in space.

In the setting of \cite{MR2083421} and \cite{MR2200881} the Bardet, Gou\"ezel and Keller \cite{MR2293068} prove the central limit theorem and the local limit theorem for Lipschitz functions depending on finitely many coordinates.  The proof of the local limit theorem requires the additional assumption that for any compact interval $J$
\[
\sigma \sqrt{2 \pi n} \mu_\epsilon \left\{  
x: \sum_{k = 0}^{n-1} f\circ T^k_\epsilon (x) \in J
 \right\} \to |J|.
\] $\sigma$ here is the variance in the central limit theorem. As in \cite{MR2083421, MR2200881}, the transfer operator is defined on the Banach space of measures of bounded variation with absolutely continuous finite dimensional marginals.

In this paper we continue the development of transfer operator techniques for studying existence and uniqueness of invariant measures corresponding to an initial potential in an appropriate class of potentials for expanding maps on a lattice coupled by an infinite range coupling. As in the setting of \cite{MR2083421, MR2200881}, we obtain a spectral gap for the transfer operator in an appropriate space of densities. In contrast to the setting of \cite{MR2083421, MR2200881}, we obtain our potentials and densities from a class of H\"older functions, and admit spatial couplings more general that those previously considered. We do not assume the existence of a reference measure; the theory is developed with respect to suitable potentials. Finally, we only require that the minimal expansion $\eta^{-1} > 1$ and the coupling strength $\epsilon > 0$ be related as $\epsilon \eta <1$ for the existence, and uniqueness, of the invariant probability measure.  We then use an abstract result of Gou\"ezel to obtain the almost sure invariance principle for our system. We also show that the invariant probability measure $\nu$ is conformal on open sets.

\section{Main results.}

In this section, we will list the main theorems that we prove. The precise definitions for the objects that appear here will follow in the body of the paper.

We start with a result known in the dynamical systems literature as ``decay of correlations''. We establish this result by showing that an appropriate transfer operator $\mathcal{L}$ has a spectral gap. For more details on the operator $\mathcal{L}$ see sections $\ref{sec:LasotaYorkeL}, \ref{sec:E}$ and $\ref{sec:mathL}$. We also check that the Lasota-Yorke inequality is stable under perturbation and in doing so, establish the almost sure invariance principle by using \cite{Go2010}. We use the space $\mathcal{C}$ of H\"older continuous functions with the norm
$\|\cdot\|=|\cdot|_\infty+|\cdot|_\beta$ where $|\cdot|_\beta$ is the $\beta$-H\"older constant (the precise definition
is below).

\begin{thm} Let $T$ be the coupled lattice map on the lattice system $\Omega$ and $\mathcal{C}$ the space 
of H\"older continuous functions on $\Omega$.

Then for any potential function $f\in \mathcal{C}$ there exists an invariant measure $\nu$ which is $g$-conformal
where $g=f+h-f\circ T$ for an $h\in\mathcal{C}$. Moreover 
there exists a constant $0<\varsigma<1$ with the property that for any $\Phi_1, \Phi_2\in\C$ there exists a constant $C_1$ such that 
\[
\left|
\int \Phi_1 \circ T^n \Phi_2 d\nu - \int \Phi_1 d\nu \int \Phi_2 d\nu
\right| 
\leq C_1 |\Phi_1|_\infty \|\Phi_2\| \varsigma^n.
\]
\end{thm}

Finally, by using an abstract result from \cite{Go2010}, and Theorem \ref{t:LY-perturb}, we prove that

\begin{thm}
The system satisfies the almost sure invariance principle.
\end{thm}

Various statistical properties such as the law of iterated logarithms, the weak invariance principle and the central limit theorem now follow as corollaries.

\begin{corollary} (CLT)
For observables $\Phi\in\mathcal{C}$ one has
$$
\mathbb{P}\left(\frac{\sum_{j=0}^{n-1}\Phi\circ T^j-n\nu(\Phi)}{\sigma\sqrt{n}}\le t\right)
\rightarrow N(t)
$$
as $n\rightarrow\infty$ where $T$ is the coupled lattice map  and 
$N(t)$ is the normal distribution.
\end{corollary}

A technical result along the way is the following proposition:

\begin{proposition} 
Let $\Phi \in \C$. Then $\LL(\Phi) \in \C$ and  moreover there exists a constant $\theta_1\in (0,1)$ and a constant $C_2>0$ such that
\[
|\LL^n(\Phi)|_\beta \leq |\Phi|_\beta \theta_1^n + C_2 |\Phi|_\infty
\]
\end{proposition}

The proposition is used to prove the existence of the invariant probability measure
$\nu$ which satisfies $\mathcal{L}^* \nu = \nu$. This result is proved as Theorem \ref{t:LY-perturb}. As a corollary to this theorem, with suitable modifications to the classical techniques we prove that the essential spectral radius of $\mathcal{L}$ is strictly smaller than the spectral radius of $\mathcal{L}$ (which is 1). 

Finally, we make a note regarding constants. We denote `global'  constants by $C_1, C_2,\dots$ throughout the 
paper and `local' constants by $c_1,c_2,\dots$ for each lemma or theorem.  The constant $C_E$
tunes the strength of interactions between the maps on the lattice.

\section{The uncoupled system.}

First, we need to define what the admissible class of observables is.
In order to study observables (and potentials) on an infinite lattice, we need to define how well approximated these observables are by restrictions to finite sub-lattices. Let $\Lambda = \Z$ and $\Omega=I^\mathbb{Z}$ 
the lattice space where $I$ is the unit interval. Define
\[
\Lambda_k = (-k, -k+1, \dots, -1, 0, 1, \dots, k-1, k).
\] 
Assume the following:

\subsection{Assumptions on the metric}
\begin{enumerate}[($d$1)]
\item Let $d_I:I\to\R^+$ be a metric on the unit interval $I$ and pick $\theta\in (0,1)$. Define a metric $d$
on $\Omega$ by
\begin{equation}
d(x, y) = \sup_{k\in\mathbb{Z}} \theta^{|k|} d_I(x_k, y_k).
\end{equation}
where $x=(\dots,x_{-1},x_0,x_1,x_2,\dots), y=(\dots,y_{-1},y_0,y_1,y_2,\dots)$ are points in $\Omega$.
\end{enumerate}


\subsection{Assumptions on the map on $I$}

\noindent The global map on $\Omega$ is thought to be composed of individual maps on the nodes $I$. We list below the conditions that our map satisfies (we note that these are not the most general conditions under which our theorem can be stated, but in order to keep our reasoning transparent, we do not pursue a more general setting).

\begin{enumerate}[($\tau$1)]
\item\label{tau1} The map $\tau:I\to I$ has full branches. As a consequence, for each $x\in I$, $b=\#\set{\tau^{-1}x}$ is constant. This also implies that the map $\tau$ has at least one fixed point. Denote one such fixed point by $p_\tau.$

\item\label{tau2} The map $\tau:I\to I$ is expanding. This is taken to mean that there exists a constant $\eta\in (0,1)$
such that for every inverse branch $\zeta$ of $\tau$ one has $d(\zeta_x, \zeta_y)\leq \eta d(x, y)$
($\zeta_x,\zeta_y$ are the images of $x,y$ under $\zeta$, i.e. $\tau\zeta_x=x,\tau\zeta_y=y$).
\end{enumerate}

Define $\ii_k: \Omega_k\to\Lambda$ as 
$\ii_k(x) = 
\begin{cases}x_i &\text{ if $|i|\leq k$}\\
 p_\tau &\text{otherwise} \end{cases}$ and put $\pi_k:\Lambda\rightarrow\Lambda$ for 
 the ``projection'' which is given by $\pi_k x=\ii_k(x|_{\Omega_k})$
 where $\Omega_k=I^{\Lambda_k}$.
 
 For $\Phi\in C(\Omega)$ we define the ``restriction'' to $\Omega_k $ as 
\[
\Phi_k = \Phi\circ\pi_k
\] 
that is $\Phi_k(x) = \Phi\circ \ii_k(x|_{\Omega_k})$. A consequence of the fact that outside of the lattice 
$\Omega_k $ we have chosen $\pi_k(x)$ to be equal to $p_\tau$ is that
\[
\bar{\tau}^{i} \pi_k(x) = \pi_k\bar\tau^{i}(x).
\]

\subsection{Definition of the function space $\mathcal{C}$}
For $\Phi\in C(\Omega)$ we define the  H\"older constant
$$
|\Phi|_\beta:= \sup_{k\in \mathbb{N}} |\Phi_k|_\beta
$$
where
\[
|\Phi_k|_\beta:= \sup_{x, y\in\Omega}\frac{|\Phi_k( x) - \Phi_k(y)|}{d (x, y)^\beta}.
\] 
Then 
$$
\|\cdot\|=|\cdot|_\infty+|\cdot|_\beta
$$
defines a norm and we define the function space 
$$
\mathcal{C}_\beta=\left\{\Phi\in C(\Omega): \|\Phi\| < \infty\right\}.
$$
We will also sometimes need the variation semi-norm for some $\alpha\in(0,1)$ given by
\[
V_\alpha(\Phi):= \sup_{k\in\mathbb{N}} \frac{\var_k(\Phi)}{\alpha^k}
\]
where
$$
\var_k(\Phi)=\sup_{x\in \Omega} \left| \Phi(x) - \Phi_k(x)\right|
$$
is the {\em $k$th variation} of $\Phi$. If  $\alpha \geq \theta^\beta$ then $V_\alpha(\Phi) \leq |\Phi|_\beta$ 
for all $\Phi\in\mathcal{C}_\beta$. Therefore, in what follows, we fix some $\alpha \geq \theta^\beta$ and, instead of writing $\mathcal{C}_\beta$, we only write $\mathcal{C}$.

\subsection{Defining the positive operators $P$ and $L$}


Let $f\in\mathcal{C}$ be a potential function and define 
for the finite sub-lattice $\Lambda_k$ the transfer operator $P_k$ for $\bar\tau|_{\Omega_k}$ by
\[
P_k(\Phi)(x) =\frac1{b_k}\sum_{|\zeta|=k} \exp(f(\ii_k \zeta_x)) \Phi(\ii_k\zeta_x)
\]
$x\in\Omega$, where the summation is over inverse branches $\zeta$ of $\bar\tau|_{\Omega_k}$, i.e.\
$\bar\tau\circ\zeta$ is the identity and moreover $\zeta_x$ the image of $x$ under $\zeta$ has 
the property that $(\zeta_x)_j=x_j$ for $|j|>k$. For the normalising factor one has
$b_k= \#\set{\bar{\tau}^{-1}x|_{\Omega_k}}=b^{2k+1} $ as  $\#\set{\tau^{-1}x}= b$.

Clearly, $P_k$ is a well defined, positive, bounded linear operator on functions on $\Omega$.
The following lemma serves to define the transfer operator $P$ by taking a limit $k$ to infinity.

\begin{lemma}\label{PkCauchy}
Let $f:\Omega\rightarrow\R$ such that $|f|_\infty+V_\alpha(f) < \infty$. Then\\
(I) $P_k$ is uniformly (in $k$) bounded.\\
(II) For $\Phi$ such that $|\Phi|_\infty+V_\alpha(\Phi)<\infty$ the sequence $(P_k(\Phi))$ is Cauchy for each $x\in\Omega$.\\
(III) $V_\alpha(P_k\Phi)\le C_3|P\Phi|_\infty\;\;\forall k$ for some constant $C_3$.
\end{lemma}

\begin{proof}
(I) Clearly  $|P_k\Phi|_\infty\le e^{|f|_\infty} |\Phi|_\infty$, which implies $P_k$ is bounded uniformly in $k$.

(II) For $k_1 < k_2$ one has 
\begin{eqnarray*}
|P_{k_2}(\Phi)( x) - P_{k_1}(\Phi)( x)| & = & \left|
\sum_{|\zeta'|=k_2} \frac{\exp(f \zeta'_x)\Phi(\zeta'_x)}{b_{k_2}} 
- \sum_{|\zeta|=k_1}\frac{ \exp( f \zeta_x) \Phi(\zeta_x)}{b_{k_1}}
\right|
\end{eqnarray*}
where $\zeta'$ are inverse branches in $\bar\tau^{-1}|_{\Omega_{k_2}}$ and 
 $\zeta$ are inverse branches in $\bar\tau^{-1}|_{\Omega_{k_1}}$.

The lattice $\Lambda_{k_2}$ contains $2(k_2 - k_1)$ elements more than $\Lambda_{k_1}$, and so 
$b_{k_2} = b^{2(k_2 - k_1)}b_{k_1}$. Therefore the above sum simplifies as
\begin{eqnarray}
\lefteqn{\left|
\sum_{|\zeta'|=k_2} \frac{\exp( f \zeta'_x)\Phi(\zeta'_x)}{b_{k_2}} 
-\sum_{|\zeta=k_1} \frac{ \exp( f \zeta_x) \Phi(\zeta_x)}{b_{k_1}}
\right|} &&\nonumber\\
& \leq &
\sum_{|\zeta|=k_1}\left[
\frac{\left(
\sum_{\zeta'|_{\Omega_{k_1}} = \zeta} \exp(f \zeta'_x) \Phi(\zeta'_x)
\right)}{b_{k_2}} - \frac{\exp(f \zeta_x)\Phi(\zeta_x)}{b_{k_1}}
\right]\nonumber\\
& = &\sum_{|\zeta|=k_1}\frac{\exp( f \zeta_x)\Phi(\zeta_x)}{b^{2k_1+1}}
\left[
\left(
\sum_{\zeta'|_{\Omega_{k_1}} = \zeta} \exp(f\zeta'_x -  f \zeta_x) \frac{\Phi(\zeta'_x)}{\Phi(\zeta_x)}
\right)\frac{b^{2k_1+1}}{b^{2k_2+1}} - 1
\right]
\end{eqnarray}

Since  $\exp(f\zeta'_x - f\zeta_x)\le e^{V_\alpha(f)}\leq 1+V_\alpha(f) \alpha^{k_1} + o(\alpha^{k_1})$ and 
\[
\left|
\frac{\Phi(\zeta'_x)}{\Phi(\zeta_x)} -1
\right| \leq \Phi(\zeta_x)V_\alpha(\Phi) \alpha^{k_1} \leq |\Phi|_\infty V_\alpha(\Phi) \alpha^{k_1}.
\] 
Therefore (for some $c_1$)
\[
\left(
\sum_{\zeta'|_{\Omega_{k_1}} = \zeta} \exp(f\zeta'_x - f \zeta_x) \frac{\Phi(\zeta'_x)}{\Phi(\zeta_x)}
\right)  = \sum_{\zeta'|_{\Omega_{k_1}} = \zeta}\left(1+ c_1\alpha^{k_1} +o(\alpha^{k_1})\right)  = b^{2k_2 - 2k_1}(1+c_1\alpha^{k_1} + o(\alpha^{k_1}))
\] from where it follows that
\[
\left[
\left(
\sum_{\zeta'|_{\Omega_{k_1}} = \zeta_x} \exp(f\zeta'_x - f \zeta_x) \frac{\Phi(\zeta'_x)}{\Phi(\zeta_x)}
\right)\frac{b^{2k_1+1}}{b^{2k_2+1}} - 1
\right] 
\leq
c_1 \alpha^{k_1} + o(\alpha^{k_1}).
\]
Therefore, 
\[
|P_{k_2}(\Phi)( x) - P_{k_1}(\Phi)(x)| \leq c_2 |P_{k_1} \Phi|_\infty \alpha^{k_1}\le c_3\alpha^k
\]
for some constants $c_2,c_3$ as by part~(I) $ |P_{k_1} \Phi|_\infty$ is uniformly bounded.

(III) This follows from the first inequality in the last line of estimates.
\end{proof}

Since by Lemma \ref{PkCauchy}, for each $x\in\Omega$, the sequence $P_k(\Phi)(x)$ is a Cauchy sequence (of real numbers) we now define the operator $P$ for the infinite lattice system as the pointwise limit:
\[
P(\Phi)(x) = \lim_{k\to\infty} P_k(\Phi)(x).
\]

\begin{lemma}
$P$ is a non-negative and continuous operator on $\mathcal{C}$.
\end{lemma}

\begin{proof} Clearly $P$ is non-negative as the approximations $P_k$ are non-negative.
Since $P$ is a linear operator, it is enough to show continuity at the origin.
For  $\Phi\in\mathcal{C}$ we see that $P$ is a bounded operator in the $|\cdot|_\infty$-norm as
\[
|P(\Phi)| \leq e^{|f|_\infty}|\Phi|_\infty.
\]
Now let $x,y\in\Omega$, then
\begin{eqnarray*}
|P\Phi(x)-P\Phi(y)|&\le&\lim_k\frac1{b_k}\sum_{|\zeta|=k}\left|
 e^{f(i_k \zeta_x)} \Phi (i_k \zeta_x) -e^{f(i_k \zeta_y)} \Phi (i_k \zeta_y) \right|\\
 &\le&\lim_k\frac1{b_k}\left(\sum_\zeta
 e^{f(i_k \zeta_x)}| \Phi (i_k \zeta_x) - \Phi (i_k \zeta_y)|
+|\Phi|_\infty\sum_\zeta e^{f(i_k \zeta_y)}\left|  e^{f(i_k \zeta_y)-f(i_k \zeta_y)}-1\right|\right)\\
 &\le&e^{|f|_\infty}|\Phi|_\beta d(x,y)^\beta+c_1|\Phi|_\infty |f|_\beta d(x,y)^\beta
\end{eqnarray*}
for some $c_1$ as $d(i_k\zeta_x,i_k\zeta_y)\le c_2d(x,y)$ ($c_2>0$).
 Hence $|P\Phi|_\beta \le c_3(|\Phi|_\beta +|\Phi|_\infty)=c_3\|\Phi\|$
for some constant $c_3$ which is independent of $\Phi$.
\end{proof}

Note that the space $\Omega$ with the metric $d$ is convex, compact and separable. Separability 
follows from the fact that for every $k\in\mathbb{N}$ there are finitely many points 
in $\ii_k \Omega_k $ that are $\vartheta^k$-dense in $\Omega$ for any $\vartheta\in(0,1)$.
In this way one produces a countable dense set in $\Omega$. 
This implies that the set $\mathscr{M}$ of probability measures on $\Omega$ is 
compact in the weak* topology. Thus following~\cite{MR2423393} we can define an operator
$\mathcal{M}:\mathscr{M}\rightarrow\mathscr{M}$ by $\mathcal{M}\nu=\frac{P^*\nu}{P^*\nu(\mathds1)}$
where $\nu\in\mathscr{M}$. By the theorem of Schauder-Tychonoff $\mathcal{M}$ has 
a fixed point $\nu$ in $\mathscr{M}$. Thus $P^*\nu=\lambda\nu$, where $\lambda=P^*\nu(\mathds1)$.

\begin{defn}
Let $f\in\mathcal{C}$ and put $B(z) = \exp\left({|f|_\beta \frac{\eta^\beta}{1 -\eta^\beta} z^\beta}\right), z\ge0$.
Define the function set
\[
\Delta_{f}:= \set{\Phi\in C(\Omega): \Phi\ge0, \nu(\Phi)=1, \Phi(x) \leq B(d(x, y)) \Phi(y)\forall x,y\in\Omega}.
\] 
\end{defn}

Notice that $B(\eta z)e^{|f|_\beta \eta^\beta z^\beta}=B(z)$.

\begin{lemma} \label{delta-is-in-C}
$\Delta_f \subset \mathcal{C}.$
\end{lemma}

\begin{proof} For $\Phi\in\Delta_f$ one has $\Phi(x)/\Phi(y)  \leq B(d(x,y))$ and $\Phi(y)/\Phi(x) \leq B(d(x, y))$
which implies
\[
|\Phi(x) - \Phi(y)|= \left|
\frac{\Phi(x)}{\Phi(y)} -1
\right| |\Phi(y)| \leq |\Phi(y)| |B(d(x,y))-1| \leq |\Phi|_\infty \left[
d(x, y)^\beta |f|_\beta  \frac{\eta^\beta}{1 -\eta^\beta}+ o(d(x, y)^\beta)
\right]
\] 
where $|\Phi|_\infty\le B(1)$ as the diameter of $\Omega$ is equal to $1$ and $\nu(\Phi)=1$
(i.e. $\inf \Phi\le1$).
This implies that $|\Phi|_\beta  \leq |f|_\beta \left(\frac{\eta^\beta}{1 -\eta^\beta}+c_1\right) < \infty$
for some $c_1$. Hence $\|\Phi\|<\infty$.
\end{proof}

In order to apply the theorem of Schauder-Tychonoff we must first show that the operator
 $\frac{1}{\lambda} P$ maps $\Delta_f$ into itself.

\begin{lemma}
$\frac{1}{\lambda} P$ maps $\Delta_f$ into itself.
\end{lemma}

\begin{proof}
Clearly $\nu(\frac1\lambda P\Phi)=1$ for all $\Phi\in\Delta_f$. Since $P$ is a positive operator 
we also get $\frac1\lambda P\Phi\ge0$ for all $\Phi\in\Delta_f$. It remains to verify the regularity 
property. Since 
$$
f(i_k\zeta_x)\le f(i_k\zeta_y)+|f|_\beta d(i_k\zeta_x,i_k\zeta_y)^\beta
\le f(i_k\zeta_y)+|f|_\beta \eta^\beta d(x,y)^\beta
$$
one obtains
\begin{eqnarray*}
\frac1\lambda P\Phi(x) & = & \frac{1}\lambda \lim_{k\to\infty} \frac{1}{b_k} \sum_{|\zeta|=k} e^{f(i_k\zeta_x)}\Phi(i_k\zeta_x)\\
&\leq & \frac{1}\lambda \lim_{k\to\infty} \frac{1}{b_k} \sum_{|\zeta|=k} e^{f(i_k \zeta_y)} \Phi(i_k \zeta_y) B(\eta d(x, y))e^{|f|_\beta \eta^\beta d(x, y)^\beta}\\
&\leq & \frac{1}\lambda P\Phi(y) B(\eta d(x, y)) e^{|f|_\beta \eta^\beta d(x, y)^\beta}\\
&\leq & \frac{1}\lambda P\Phi(y)B(d(x, y)).
\end{eqnarray*}
Therefore, $\frac{1}\lambda P\Phi\in \Delta_f.$

\end{proof}

\begin{lemma}
There exist a unique $h\in\Delta_f$ so that $Ph=\lambda h$ and moreover $h$ is strictly
positive.
\end{lemma}

\begin{proof}
The set $\Delta_f$ is convex and equicontinuous by Lemma~\ref{delta-is-in-C}. Therefore $\Delta_f$ 
is compact in the $|\cdot|_\infty$-norm by Arzela-Ascoli and $\frac1\lambda P$ has by Schauder-Tychonoff
a fixed point $h\in\Delta_f$. That is $Ph=\lambda h$. To see that $h$ is strictly positive assume 
that $h$ has a zero at $x\in\Omega$, i.e. $h(x)=0$. Then 
$$
0=\frac1{\lambda^n}P^n h(x)=\frac1{\lambda^n}\lim_kP_k^n h(x)
=\frac1{\lambda^n} \lim_{k\to\infty} \frac{1}{b_k} \sum_{\zeta} e^{f^n(i_k\zeta_x)}h(i_k\zeta_x)
$$
where the sum is over all inverse branches $\zeta$ of $\bar\tau^{-n}$ in $i_k\Lambda_k$.
Since $h\ge0$ this implies that $h(\zeta_x)=0$ for all inverse branches $\zeta$ of $\bar\tau^n$.
Since the set $\bigcup_n\bigcup_{\zeta\in\bar\tau^{-n}} \zeta_x$ is dense in $\Omega$ 
and $h$ is continuous we conclude that $h$ is identically zero which contradicts the 
assumption $\nu(h)=1$.

To obtain uniqueness of $h$ assume that there is a second eigenfunction $h'\in\Delta_f$ so that 
$Ph'=\lambda h'$ and put $t=\inf\frac{h'}h$. By convexity of $\Delta_f$ one has 
$h-th\in\Delta_f$ and by choice of $t$ there exists an $x\in\Omega$ so that $(h-th')(t)=0$.
By the argument above we conclude that $h-th'$ must vanish identically, which is impossible.
Hence $h$ is unique.
\end{proof}

We now define the normalized transfer operator $L:\mathcal{C}\rightarrow\mathcal{C}$ by 
putting $L(\Phi):= P(h \Phi) / (\lambda h)$ for $\Phi\in\mathcal{C}$. Note that $L(\Phi)$ is well defined since $h>0$ 
and has the potential function $g=f-\log\lambda-\log h\circ\bar\tau+\log h$.
Moreover $L$ has the (dominant) simple eigenvalue $1$ with eigenfunction $\mathds1$ as 
 $L(\mathds1) = \mathds1$. The associated eigen-functional $\mu=h\nu$ is a $\bar\tau$-invariant
 probability measure. Define, also, $L_k (\Phi)$ as follows:
\[
L_k(\Phi)(x) = \frac{1}{b_k} \sum_{|\zeta|=k} e^{g(i_k \zeta_x)} \Phi(i_k \zeta_x)
\] where $g(x) = f(x) - \log \lambda - \log h\circ \bar\tau (x)+\log h (x).$ Note that by definition, 
$|g|_\beta \leq |f|_\beta + 2|h|_\beta < \infty$, and $|g|_\infty<\infty.$

Hence, we state a corollary to Lemma \ref{PkCauchy}:

\begin{corollary}
For each $x$, and for each $\Phi$, the sequence $L_k(\Phi)(x)$ is Cauchy, and hence it converges to $L(\Phi)(x)$.
\end{corollary}

\begin{proof}
The fact that $L_k(\Phi)(x)$ is Cauchy, and hence has a point wise limit, follows directly from 
Lemma~\ref{PkCauchy}. Also, notice that
\begin{eqnarray*}
L(\Phi)(x) & =& \frac{P(h\Phi)(x)}{\lambda h(x)}\\
& = & \lim_{k\to\infty} \frac{P_k(h\Phi)(x)}{\lambda h(x)}\\
& \geq& \lim_{k\to\infty} \frac{1}{b_k} \sum_{|\zeta|=k}
e^{f(i_k\zeta_x)}h(i_k\zeta_x) \Phi(i_k\zeta_x)\frac{1}{\lambda h(x)}.
\end{eqnarray*}
The term in the last summation can be written as
\[
e^{f(i_k\zeta_x)}h(i_k\zeta_x)\Phi(i_k\zeta_x)\frac{1}{\lambda h(x)} 
= e^{g(i_k \zeta_x)}\Phi(i_k\zeta_x)\frac{h(\tau(i_k\zeta_x))}{h(x)},
\] 
and so, because $\var(h) < \infty$, for any $\epsilon>0$, we have that
\[
L(\Phi)(x) \geq (1 - \epsilon) \lim_{k\to\infty} L_k(\Phi)(x).
\]
Similarly, we obtain that $L(\Phi)(x) \leq (1+\epsilon)\lim_{k\to\infty} L_k(\Phi)(x).$ This completes the proof.
\end{proof}

\subsection{Lasota-Yorke inequality for $L$} \label{sec:LasotaYorkeL}

We now establish a Lasota-Yorke inequality for the operator $L$. We do this by obtaining the corresponding inequality for each approximation $L_k$. 
We use the following notation: For a function $\Phi$ we denote by 
$\Phi^{(n)}=\sum_{j=0}^{n-1}\Phi\circ\bar\tau$ its $n$th ergodic sum.
We shall need the following technical estimate.

\begin{lemma}\label{lemma:technical}
Let $f\in\mathcal{C}$ be a potential and let $g = f - \log \lambda - \log h\circ \bar\tau+\log h.$ 

Then for all $x$ and $y\in\I$:
\[
\left|
g^{(n)}(i_k\zeta_{x}) - g^{(n)}(i_k\zeta_{y})
\right|
\leq  \left(3 |h|_\beta +\frac{ \eta^\beta}{1-\eta^\beta} |f|_\beta \right)d(x, y)^\beta .
\]
\end{lemma}

\begin{proof}
Since
\[
g^{(n)}(i_k \zeta_x) = \sum_{j = 0}^{n-1} g\circ \bar\tau^{j}(i_k\zeta_x) = \sum_{j =0}^{n-1} f\circ \bar\tau^{j}(i_k\zeta_x) - n\log\lambda  + h(i_k\zeta_x) - h\circ\bar\tau^n(i_k\zeta_x)\]
we get that 
\[
g^{(n)}(i_k\zeta_x) - g^{(n)}(i_k\zeta_y) = f^{(n)}(i_k\zeta_x) - f^{(n)}(i_k\zeta_y) + h(i_k\zeta_x) - h(i_k\zeta_y) + h\circ \bar\tau^{n}(i_k\zeta_y) - h\circ \bar\tau^{n}(i_k\zeta_x).
\]
Note that $|h\circ \bar\tau^{n}(i_k\zeta_y) - h\circ \bar\tau^{n}(i_k\zeta_x)|$ is bounded by 
$|h|_\beta d(i_k x|_{\Lambda_k}, i_k y|_{\Lambda_k})^\beta \leq |h|_\beta d(x, y)^\beta$
and  $|h(i_k\zeta_x) - h(i_k\zeta_y)|$ is bounded by $|h|_\beta \eta^{\beta n} d(x, y)^\beta$.
Define $\tau_{(k)}:= \bar \tau \circ \mathfrak{i}_k$. The above combined with 
\begin{eqnarray*}
\left|f^{(n)}(i_k\zeta_{x}) - f^{(n)}(i_k\zeta_y)\right|
&\leq & \sum_{i = 0}^{n-1} \left|f\circ \bar \tau^{i}(i_k\zeta_{x}) - f \circ \bar\tau^{i}(i_k\zeta_{y})\right|\\
&\leq & \sum_{i = 0}^{n-1} \left| f (i_k (\tau_{(k)}^{i}\zeta_{x})) - f (i_k (\tau_{(k)}^{i}\zeta_{ y}))\right|\\
&\leq & \sum_{i = 0}^{n-1}
\left| f (i_k ((\tau^{i}i_k\zeta_{x}))|_{\Omega_k}) - f (i_k ((\tau^{i}i_k\zeta_{y}))|_{\Omega_k})\right|\\
&\leq & |f_k|_\beta \sum_{i = 0}^{n-1}d(i_k(\bar \tau^{i}_{(k)} \zeta_{x}), i_k(\bar \tau^{i}_{(k)} \zeta_{y}))^\beta\\
&\leq & |f|_\beta d(x,  y)^\beta \sum_{i = 0}^{n-1} \eta^{\beta(n -i)}.
\end{eqnarray*}
proves the desired bound.
\end{proof}

\begin{proposition}[Lasota-Yorke inequality for $L_k$]\label{lemma:LYlk}
Let $\Phi\in\mathcal{C}$ and $f, h, \tau$ be as before. Let $k\in\mathbb{Z}^+$. Then $\exists \;C_4>0$ 
(depending on $f, h$ and $\tau$) such that
\[
|L^n_k(\Phi)|_\beta \leq \left(|\Phi|_\beta  \eta^{\beta n} + C_4|\Phi|_\infty \right) L_k^n(\mathds1).
\]
\end{proposition}

\begin{proof}
For $x, y\in \I.$:
\begin{eqnarray*}
|L^n_k(\Phi)(x) - L^n_k(\Phi)(y)| & \leq  &\frac{1}{b_k} \sum_{|\zeta|=}   \left|    e^{g^{(n)}(i_k\zeta_x)}\Phi(i_k \zeta_x) - e^{g^{(n)}(i_k\zeta_y)}\Phi(i_k\zeta_y) \right| \\
&\leq & \frac{1}{b_k} \left[ 
\sum_\zeta e^{g^{(n)}(i_k \zeta_x)} \left|   \Phi(i_k\zeta_x) - \Phi(i_k\zeta_y)     \right|     +      \sum_\zeta |\Phi|_\infty \left|  e^{g^{(n)}(i_k\zeta_x)} - e^{g^{(n)}(i_k\zeta_y)}   \right|
\right] \\
&\leq & \frac{1}{b_k} \left[ 
\sum_\zeta |\Phi|_\beta  \eta^{\beta n} d(x, y)^\beta L_k^n(\mathds1) \right.\\
&&\hspace{2cm}+ \left.   \sum_\zeta |\Phi|_\infty e^{g^{(n)}(i_k\zeta_y)} \left|    1 - \exp(g^{(n)}(i_k\zeta_x) - g^{(n)}(i_k\zeta_y))        \right|
  \right]
\end{eqnarray*}
as $d(i_k\zeta_x,i_k\zeta_y)\le\eta^nd(x,y)$.
By Lemma~\ref{lemma:technical}, we have 
\[
\left|  1 - \exp(g^{(n)}(i_k\zeta_x) - g^{(n)}(i_k\zeta_y))      \right|  
\leq c_1 d(x, y)^\beta + o(d(x, y)^\beta)
\leq c_2 d(x, y)^\beta
\] 
where $ c_1\le3 |h|_\beta  +\frac{ \eta^\beta}{1-\eta^\beta}|f|_\beta $. Consequently
\[
\sup_{x\neq y}\frac{|L_k^n(\Phi)(x) - L_k^n(\Phi)(y)|}{d(x, y)^\beta} 
\leq |\Phi|_\beta \eta^{\beta n} L_k^n(\mathds1) + c_2|\Phi|_\infty L^n_k(\mathds1).
\]
\end{proof}

Now we can prove the 'Lasota-Yorke' inequality for the operator $L$ for the uncoupled map $\bar\tau$ on the full, infinite lattice $\I$.

\begin{thm}
Let $\Phi\in\mathcal{C}$ and let $f, h, \tau$ be as before, and $n\in\mathbb{N}$. There exists a constant $C_5>0$ 
depending only on $f, h, \tau$ such that 
\[
|L^n(\Phi)|_\beta \leq |\Phi|_\beta \eta^{\beta n}+ C_5 |\Phi|_\infty.
\]
\end{thm}

\begin{proof} Recall $\eta$ from Assumption ($\tau$\ref{tau2}). 
Let $k_2\geq k_1\in \mathbb{Z}^+$. Then for all $x, y\in\I$ one has
\begin{eqnarray*}
|L_{k_2}^n(\Phi)(y) - L_{k_1}^n(\Phi)(x)| &\leq & |L_{k_1}^n(\Phi)(x) - L_{k_1}^n(\Phi)(y)| + |L_{k_1}^n(\Phi)(y) - L_{k_2}^n(\Phi)(y)| \\
&\leq & \left[ |\Phi|_\beta \eta^{n\beta} + C_4 |\Phi|_\infty \right] L_{k_1}^n(\mathds1) d(x, y)^\beta+ C_3 |L_{k_1}^n(\Phi)|_\infty \alpha^{k_1}\\
&\leq &\left[ |\Phi|_\beta  \eta^{\beta n} + C_4 |\Phi|_\infty \right]L_{k_1}^n(\mathds1)  d(x, y)^\beta+ C_3 e^{|g|_\infty}|\Phi|_\infty \alpha^{k_1}, 
\end{eqnarray*}
where the second line uses Lemmas~\ref{lemma:LYlk} and~\ref{PkCauchy}~(III) while 
the last line follows from the proof of Lemma \ref{PkCauchy}~(I). Letting $k_2\to\infty$ we see that
\[
|L^n(\Phi)(y) - L_{k_1}^n(\Phi)(x)| \leq \left[ |\Phi|_\beta \eta^{\beta n} + C_4 |\Phi|_\infty \right] 
L_{k_1}^n(\mathds1) d(x, y)^\beta+ C_3 e^{|g|_\infty}|\Phi|_\infty \alpha^{k_1}
\] and then on letting $k_1\to\infty$, and recalling that $L_{k_1}(\mathds1)\to L(\mathds1) =\mathds1$ we get
\[
|L^n(\Phi)(y) - L^n(\Phi)(x)| \leq \left[ |\Phi|_\beta \eta^{\beta n} + C_4 |\Phi|_\infty \right] d(x, y)^\beta
\] 
The proof is complete on dividing by $d(x, y)^\beta$ and taking the supremum. We put $C_5=C_4$.
\end{proof}

\section{The perturbation $E$.}\label{sec:E}

Denote by $\sigma:\Omega\to\Omega$ the shift map which is given by
 $\sigma((x_i)_{i = -\infty}^\infty) = (x_{i+1})_{i = -\infty}^\infty$.
Let $E:\I \to \I$ be a function such that $E^{-1}$ exists, and there exists a constant $C_E \in (0, \eta^{-1})$ 
satisfying the following Assumption (E\ref{E1}):
\begin{enumerate}[(E1)]

\item\label{E1}  For each $n\in\mathbb{Z}$ one has
\[
d(\sigma^n E^{-1}x, \sigma^n E^{-1}y) \leq C_E d(\sigma^n x, \sigma^n y).
\]

\end{enumerate}
Now define the coupled map $T:\I\to \I$ by $T(x) = E(\bar\tau(x)).$ Define, for $\Phi\in\C$, $\LL(\Phi)(x)  = L(\Phi) (E^{-1}(x)).$ It is straightforward to verify that for $\Phi, \Psi\in \C$, 
$\LL(\Phi\circ T\cdot \Psi) = \Phi \cdot \LL(\Psi)$ (as $T=E\circ\bar\tau$).

\begin{thm} \label{t:LY-perturb}
There exists a constant $C_6$ such that 
for every $\Phi \in \C$, $\LL(\Phi) \in \C$ one has 
\[
|\LL^n(\Phi)|_\beta 
\leq |\Phi|_\beta (C_E \eta)^{\beta n} + C_6 |\Phi|_\infty C_E^\beta \sum_{i = 0}^\infty (C_E \eta)^{i\beta}
\]
\end{thm}

\begin{proof} By Lemma~3.10 ($n=1$)
\begin{eqnarray*}
|\LL(\Phi)(x) - \LL(\Phi)(y)| &  = & |L(\Phi)(E^{-1}(x)) - L(\Phi)(E^{-1}(y))|\\
&\leq & |L(\Phi)|_\beta  d(E^{-1} x, E^{-1} y)^\beta\\
&\leq & \left(   |\Phi|_\beta \eta^\beta   + C_5 |\Phi|_\infty  \right) C_E^\beta d(x, y)^\beta
\end{eqnarray*}
and so
\[
|\LL(\Phi)|_\beta  \leq|\Phi|_\beta (C_E \eta )^\beta + C_5 C_E^\beta |\Phi|_\infty.
\]

Assume the formula for $n-1$. Since $|\LL(\Phi)|_\infty  = |L(\Phi)|_\infty \leq |\Phi|_\infty$, we obtain inductively
\begin{eqnarray*}
|\LL^n(\Phi)|_\beta  &\leq & |\LL^{n-1}(\Phi)|_\beta  (C_E \eta )^\beta + C_5 C_E^\beta |\Phi|_\infty\\
&\leq & |\Phi|_\beta (C_E \eta )^{n\beta} 
+ C_5 |\Phi|_\infty C_E^\beta \sum_{i = 1}^\infty (C_E \eta )^{\beta i} + C_5 C_E^\beta |\Phi|_\infty\\
&=& |\Phi|_\beta (C_E \eta )^{n\beta} 
+ C_5 |\Phi|_\infty C_E^\beta \sum_{i = 0}^\infty (C_E \eta )^{\beta i}.
\end{eqnarray*}
\end{proof}

Let us note that $C_4=C_5=C_6$. 
We record, for future use, two elementary topological properties of the map $T:\Omega\to\Omega$. 

\begin{lemma} \label{l:dense-back-orbit}
For each $x\in\Omega$, $\set{T^{-n}x, n\geq 0}$ is dense in $\Omega.$  
\end{lemma}

\begin{proof}
Let $x\in\Omega$, and let $\epsilon > 0.$ Since $\set{\tau^{-1}x_0}\cup \set{0,1}$ forms a partition of $[0,1]$ with mesh $< \eta$, $\bar\tau^{-1}(E^{-1} x)|_{\Lambda_0}$ also forms a partition with mesh smaller than $C_E\eta < 1$. On iterating, we see that $T^{-n}(x)|_{\Lambda_0}$ generates a partition with mesh less than $(C_E\eta)^n.$ This remains true on each node of the lattice; denoting by $\Delta_i$ the mesh of the $i$th node of the lattice and defining $\Delta = \sup_{i} \theta^{|i|} \Delta_i $, we see that there will be a pre-image of $x$ within $\epsilon$ of an arbitrarily chosen point $y$ provided $n\geq \log(1/\epsilon)/ \log (1/C_E\eta)$.
\end{proof}

\begin{lemma} \label{l:top-trans}
If $U$ and $V$ are any open subsets of $\Omega$, then there exists an $N\in\mathbb{Z^+}$ such that for all $n\geq N$, $U\cap T^n V \neq \emptyset.$
\end{lemma}

\begin{proof}
Let $x\in U$. Since $\set{T^{-n}x, n\geq 0}$ is dense in $\Omega$, there exists an $n_1$ such that $T^{-n_1}x\in V.$ Since $V$ is open, there exists a $\delta>0$ such that $B_\delta(T^{-n_1}x)\in V.$ Choose $n_2  \geq \log(2/\delta)/\log(1/C_E\eta)$. Then for $n\geq n_2$, there always exists an element of $\set{T^{-n}x}$ inside $B_\delta(T^{-n_1}x)$ since the mesh $\Delta$ as defined above is less than $\delta$ for $n\geq n_2.$ Hence $T^{-n} U \cap V \neq \emptyset \;\forall \; n\geq n_2.$
\end{proof}

We note that the above lemmas imply that the map $T$ is topologically mixing, and since $\Omega$ is compact, forward transitive. Further, the map $T$ is expansive, namely, if $x\in\Omega$ and $y\in\Omega$ are such that $T^n x = T^n y$ for every $n\in\mathbb{N}$ then $x = y.$

\section{Spectral properties of $\mathcal{L}$.} \label{sec:mathL}

\subsection{Ionescu-Tulcea and Marinescu Theorem.}
In this section we recall the classical machinery of quasi-compactness that is used to establish the statistical properties of a deterministic dynamical system, once a Lasota-Yorke type inequality (see Theorem \ref{t:LY-perturb}) can be established. The main ingredient is the theorem by Ionescu-Tulcea and Marinescu.

\begin{defn}
Let $L$ be an operator on a Banach space $(V, \| \cdot \|)$. $L$ is {\em quasi-compact} if there exists a positive integer $r$ and a compact operator $K$ such that $\| L^r -K \| <1.$
\end{defn}

If $L$ is quasi-compact, then $V = F \oplus H $ with $F$ and $H$ invariant under $L$, $\dim F <\infty$, $r(L|_H) < r(L)$ and each eigenvalue of $L|_F$ has modulus $r(L)$, where $r(\cdot)$ denotes the spectral radius.

As noted in \cite{MR1461536, MR2518822}, quasi compactness can be restated in many equivalent ways. We give below the following definition \cite{MR2518822} as this is the form in which we will use quasi-compactness.

\begin{defn}[Theorem 2.5.3, \cite{MR2518822}]
$L:V\to V$ is quasi-compact if and only if there are bounded linear operators $\set{Q_{\sigma}:\sigma\in\Upsilon}$ and $R$ on $V$ such that 
\begin{eqnarray*}
L^n &=& \sum_{\sigma\in\Upsilon}\sigma^n \phi_\sigma + R^n \quad \forall \quad n = 1, 2, \dots, \\
\phi_\sigma \phi_{\sigma'} &=& 0 \text{ if } \sigma \neq \sigma'\\
\phi^2_\sigma &=& \phi_\sigma, \quad \forall \quad \sigma \in \Upsilon\\
\phi_\sigma R =R\phi_\sigma &=& 0 \quad \forall \quad \sigma\in\Upsilon\\
\phi_\sigma V &=& D(\sigma), \quad \forall \quad \sigma\in\Upsilon\\
r(R) &<&1
\end{eqnarray*}
where $\Upsilon$ is the set of the eigenvalues of $L$ with modulus 1, $D(\sigma) = \set{f\in V: Lf = \sigma f}$ is the eigen-space of $L$ corresponding to the eigenvalue $\sigma$ and $r(R):= \lim_{n\to\infty} \|R^n\|^{1/n}$ is the spectral radius of $R$.
\end{defn}

Next, we recall a version of the Ionescu-Tulcea and Marinescu theorem, established by Hennion and Herv\'e\cite[Theorem II.5]{MR1862393}.

\begin{thm}
Let $|\cdot|$ be a continuous semi-norm on a Banach space $(V, \| \cdot \|)$ and let $Q$ be a bounded operator on $V$ such that
\begin{enumerate}
\item \[
Q(\set{f: f\in V, \| f\| \leq 1})
\] is conditionally compact in $(V, |\cdot|)$
\item there exists a constant $M$ such that for all $f\in V$, $|Q(f)| \leq M |f|$
\item there exists a $k\in\mathbb{N}$ and real numbers $r$ and $R$ such that $r<r(Q)$ and for all $f\in V$
\[
\|Q^k f\| \leq R |f| + r^k \|f\|.
\]
\end{enumerate} 
Then $Q$ is quasi-compact.
\end{thm}

\subsection{Quasi-compactness and other properties of $\mathcal{L}$, $P$.}

\begin{lemma}
$\mathcal{L}$ is quasi-compact.
\end{lemma}

\begin{proof}
We will take $|\cdot | = |\cdot |_\infty$ and $\| \cdot \| = |\cdot|_\beta + | \cdot |_\infty.$ The set $\mathcal{C}$ is a Banach space under $\|\cdot\|. $ Clearly, condition 1 is satisfied since the unit ball of $(V, \| \cdot \|)$ is mapped under $\mathcal{L}$ inside a ball of finite radius. Condition 2 is true because $|\mathcal{L}(\Phi)|_\infty \leq |\Phi|_\infty.$ Condition 3 follows very easily from Theorem \ref{t:LY-perturb} and the observation that 
$\mathcal{L}\mathds1 = \mathds1$ implies that $r(\mathcal{L}) \geq 1$.
\end{proof}

The next sequence of lemmas establish that 1 is the unique eigenvalue of $\mathcal{L}$ on the unit disk, and that 1 is a simple eigenvalue. The proofs of these statements follow standard lines (see, for instance, the proof of the Ruelle-Perron-Frobenius Theorem, \cite{MR1793194}), and require that the map $T$ be forward transitive, and topologically mixing. We have to, however, account for the non-standard definition of $\mathcal{L}$.

As a technical point, we observe that the operators $\mathcal{L}_k$ defined as
\[
\mathcal{L}_k(\Phi)(x) = L_k(\Phi)(E^{-1}x)
\] are point-wise approximations to the operator $\mathcal{L}$, that is, $\lim_{k\to\infty}\mathcal{L}_k(\Phi)(x) = \mathcal{L}(\Phi)(x)$ for each $\Phi\in\mathcal{C}$ and $x\in\Omega.$

\begin{lemma}\label{l:1-simple-L}
$1$ is a simple eigenvalue for $\mathcal{L}$ with eigenfunction $\mathds1$.
\end{lemma}

\begin{proof}
First, we show that any eigenfunction $\psi$ for $\mathcal{L}$ to the eigenvalue $1$ is either zero, or nowhere vanishing. 

Since $\mathcal{L}$ is a real linear operator it is enough to consider real valued eigenfunctions.

First we show that any real valued eigenfunction can be written as a linear combination of non-negative eigenfunctions. 
Let $\psi^+$ and $\psi^-$ be the positive and negative parts of $\psi$. Since $\psi^{\pm} \leq |\psi|$, $\psi^\pm \in\mathcal{C}.$ Further, the set 
\[
F_+ := \set{\left. \frac{1}{n}\sum_{j = 0}^{n-1}\mathcal{L}^j\psi^{+} \right| n\geq 1}
\] has a bounded diameter in the Lipschitz norm and hence is equicontinuous and bounded in the $|\cdot| _\infty$ norm. The Arzel\`{a}-Ascoli theorem implies that  there exists a subsequence $n_j$ such that $\lim_{n_j\to\infty} \frac{1}{n_j} \sum_{j = 0}^{n_j-1} \mathcal{L}^j(\psi^+) \to \psi^+_\infty$ uniformly in the $|\cdot |_\infty$ norm. Clearly, $\psi^{+}_\infty\geq 0$, and $\mathcal{L}\psi^+_\infty = \psi^+_\infty.$ Finally, $|\psi^+_\infty|_\infty <\infty$ and by Theorem \ref{t:LY-perturb}, $|\psi^+_\infty|_\beta < \infty$; this implies that $\psi^+_\infty \in \mathcal{C}.$ A similar analysis can be performed for $\psi^-. $ It then follows from a straightforward diagonalization argument that 
\[
\psi = \lim_{n_j\to\infty} \frac{1}{n_j} \sum_{j = 0}^{n_j-1} \mathcal{L}^j \psi^+ - \frac{1}{n_j}\sum_{j = 0}^{n_j-1}\mathcal{L}^j \psi^{-} = \psi^+_\infty - \psi^-_\infty.
\]

For the rest of this proof, it will be assumed that all eigenfunctions are non-negative and real. Let $x\in\Omega$ be a point such that $\psi(x) = 0$, with $\psi$ an eigenfunction. Then
\[
0 = \psi(x) = \mathcal{L}^n (\psi)(x) = \lim_{k\to\infty} \mathcal{L}^n_k (\psi)(x) 
= \lim_{k\to\infty} \sum_{|\zeta|=k} \frac{1}{b_k} e^{g^{(n)}(i_k\zeta_x)} \psi( i_k\zeta_x).
\] By Lemma~\ref{l:dense-back-orbit} and the non-negativity and continuity of $\psi$, it follows that $\psi$ is identically $0$. 

To show that the geometric multiplicity of $1$ is $1$, suppose there are 
two positive real eigenfunctions $\phi$ and $\psi$  and put
\[
t = \inf_{x\in\Omega}\frac{\phi(x)}{\psi(x)},
\] 
which equals $\phi(z) / \psi(z)$ at some point $z\in\Omega$. Then
$ h(x) = \phi(x) - t \psi(x)$ is an eigenfunction  to the eigenvalue $1$, and $h(z) = 0$ which, by the 
previous paragraph, implies that $h \equiv 0.$ Therefore $\psi$ is some multiple of $\phi$.

Finally, we show that the algebraic multiplicity of $1$ is also 1. Suppose not. Then there exists a $\psi$ with $(1 - \mathcal{L})^2 \psi = 0$ but $(1 - \mathcal{L})\psi \neq 0. $ Since $(1 - \mathcal{L})\psi$ is an eigenvector for $\mathcal{L}$, we must have $\mathcal{L}\psi - \psi = k\mathds1$ for some $k\not=0$. Iteration yields
 $\mathcal{L}^n \psi = nk\mathds1 + \psi $ which contradicts the uniform boundedness of $\mathcal{L}^n$
 as $|\mathcal{L}^n\psi|_\infty\le|\psi|_\infty\;\forall n$.
\end{proof}

\begin{lemma}
Let $\mathscr{M}(\Omega)$ be the space of complex Radon measures on $\Omega$. The operator $\mathcal{L}^*:\mathscr{M}(\Omega)\to\mathscr{M}(\Omega)$ is well defined. There exists a unique probability measure $\nu$ such that $\mathcal{L}^*(\nu) = \nu.$
\end{lemma}

\begin{proof}
Note that $\mathcal{C}$ is dense in $C(\Omega)$ (by the Stone-Weierstrass Theorem, since $\mathcal{C}$ separates points, $\Omega$ is compact, Hausdorff, and $\mathcal{C}$ contains the constant functions). Since $\mathcal{L}$ is Lipschitz on $\mathcal{C}$, it has a continuous extension to $C(\Omega)$, that we also denote by $\mathcal{L}$. Since the dual of $C(\Omega)$ is $\mathscr{M}(\Omega)$, the operator $\mathcal{L}^*:\mathscr{M} \to \mathscr{M}$ is well defined.  Since $\mathcal{L}$ and $\mathcal{L}^*$ have the same spectrum, 1 is also a simple eigenvalue for $\mathcal{L}^*$. By Lemma \ref{l:1-simple-L} we conclude that there exists a probability measure $\nu$ such that $\mathcal{L}^*\nu = \nu.$

To establish the uniqueness of $\nu$ we prove that 1 is the only eigenvalue for $\mathcal{L}$ on the unit disk. Let $\lambda$ be another eigenvalue of modulus 1, and let $\phi_\lambda$ be the eigenvector corresponding to $\lambda$. By orthogonality, $\int \phi_\lambda d\nu = 0. $ But, for any $\psi\in \C$,
we will show that $\mathcal{L}^n(\psi)\to \int \psi d\nu$. This will then give us a contradiction, because $\mathcal{L}^n \phi_\lambda = \lambda^n \phi_\lambda$ does not converge to 0. 

We note that it is sufficient to prove this claim for positive $\psi\in\C$ because if $\psi'\in\C$, then on decomposing $\psi' = \psi^+ - \psi^-$ we get $\mathcal{L}^n \psi' = \mathcal{L}^n\psi'^+ - \mathcal{L}^n \psi'^- \to \int (\psi'^+ - \psi'^-) d\nu = \int \psi' d\nu.$ Now, if $\psi$ is some positive element of $\C$, denote by $\tilde \psi$ any continuous accumulation point of $\mathcal{L}^n \psi.$ We will show that $\tilde \psi$ must be constant.

To see this, observe that $\tilde \psi \geq 0$ and
\[
\sup \tilde \psi \geq \sup \mathcal{L} \tilde \psi \geq \dots \geq \sup \mathcal{L}^n \tilde \psi \geq \dots.
\] Since $\tilde \psi$ is an accumulation point for $\mathcal{L}^n \psi$ none of these inequalities can be strict, so in fact $\sup \mathcal{L}^n \tilde \psi = \sup \tilde \psi$ for all $n\geq 0.$ By continuity, there exists a point $x_n$ with $\sup \mathcal{L}^n \tilde \psi= \mathcal{L}^n \tilde \psi (x_n).$ This implies that
$$
\tilde \psi(x_0) = \mathcal{L}^n \tilde \psi (x_n)
 = \lim_{k\to\infty} \frac{1}{b_k} \sum_{|\zeta|=k} e^{g^{(n)} i_k \zeta_{x_n}} \tilde \psi(i_k \zeta_{x_n})
$$
from where it follows that
\[
\lim_{k\to\infty} \left|
\frac{1}{b_k} \sum_{|\zeta|=k} e^{g^{(n)}(i_k \zeta_{x_n})} \left(
\tilde \psi(x_0) - \tilde \psi(i_k\zeta_{x_n})
\right)
\right| = 0.
\]
Using again the fact that 
\[
\frac{1}{b_k} \sum_{|\zeta|=k} e^{g^{(n)}(i_k\zeta_{x_n})} = \mathcal{L}_k^n \mathds1(x_n) \to 1 ~\text{as}~ k\to\infty
\] 
for each $n\geq 0$, we get that, on writing 
$ \mathcal{L}_k^n \mathds1(x_n)= 1+\epsilon_k$ with $\epsilon_k\to 0$, and observing that $\tilde \psi(x_0) \geq \tilde \psi (i_k\zeta)$ for each branch, 
\[
\lim_{k\to\infty}(1+\epsilon_k) (\tilde \psi(x_0) - \tilde \psi (i_k\zeta_{x_n})) = 0.
\]
Since for each branch $\lim_{k\to\infty} i_k\zeta_{x_n} = y$ where $y$ satisfies $T^n y = x_n$ 
we have that $\tilde \psi (y) = \tilde \psi (x_0)$ for all $y \in \set{T^{-n} x_n}$ for each $n\geq 0.$ It follows from lemma~\ref{l:top-trans} that $\cup_{n} \set{T^{-n} x_n}$ is dense, and so $\tilde \psi$ is constant.

Therefore, 
\[
\tilde \psi = \int \tilde \psi d\nu = \int \mathcal{L}^{n_j} \psi d\nu = \int \psi d\nu
\] where the last equality uses the fact that $\mathcal{L}^* \nu = \nu.$
\end{proof}

We now show that the measure $\nu$ constructed above is conformal. Recall that for any measurable function $f$, a measure $\mu$ is said to be $f-$conformal if for every measurable set $A$ on which the map $T:A \to TA$ is invertible, 
\begin{equation}\label{eq:conformal}
\int_A e^{-f} d\mu = \mu(TA). 
\end{equation} In the context of continuous transformations $T$ on a compact metric space $\Omega$, when the map $T$ has a finite generating partition, one requires that \eqref{eq:conformal} hold for all measurable $A$. In our setup, we do not obtain a finite, or countable, generating partition for the map $T$. Therefore, we only check \eqref{eq:conformal} on open sets.

\begin{proposition}
The measure $\nu$ is $g$-conformal, i.e., if $A$ be an open set on which $T: A \to TA$ is invertible, then
\[
\int_A e^{-g} d\nu = \nu(TA).
\]
\end{proposition}

\begin{proof}
\begin{eqnarray*}
\int_A e^{-g} d\nu & = & \int_\Omega \mathcal{L} (\chi_A e^{-g})(x)\,d\nu(x)\\
& = & \int_\Omega \lim_{k\to\infty} (1/ b_{k})\sum_{\zeta\in T^{-1}(\pi_k x)} e^{g(i_k \zeta)} e^{-g(i_k\zeta)}\chi_A (i_k\zeta)\,d\nu(x)\\
& = & \int_\Omega \lim_{k\to\infty} (1/b_{k})\sum_{\zeta\in T^{-1}(\pi_k x)} \chi_A(i_k \zeta)\,d\nu(x)\\
& = & \int_{\Omega} \lim_{k\to\infty} (1/b_{k})\chi_{TA} (x) b_{k}\,d\nu(x)\\
& = & \nu (TA).
\end{eqnarray*} The third equality from the bottom follows by observing that $\zeta \in T^{-1} (x) \in A$ if and only if $x \in TA.$

\end{proof}

\begin{thm}\label{t:cordecay}
There exists a constant $0<\varsigma<1$ with the property that for any $\phi_1, \phi_2\in\C$ there exists a constant $C_7$ such that 
\[
\left|
\int \phi_1 \circ T^n \phi_2 d\nu - \int \phi_1 d\nu \int \phi_2 d\nu
\right| 
\leq C_7 |\phi_1|_\infty \|\phi_2\| \varsigma^n.
\]
\end{thm}

\begin{proof}
$\mathcal{L}$ has a unique eigenvalue of modulus 1, therefore the projection operator on the eigen-space of 1 is defined by $\mathcal{P}(\phi) = \int \phi d\nu$. Since $\mathcal{L}$ is quasi-compact, there must exist a constant $\varsigma \in (0,1)$ such that (for some $c_1$)
\[
\|\mathcal{L}^n (\phi - \mathcal{P}(\phi))\| \leq c_1 \varsigma^n \|\phi\|.
\] Let $\tilde \phi_1 = \phi_1 - \int \phi_1d\nu.$ Observe that 
\[
\left|
\int \phi_1 \circ T^n \phi_2 d\nu - \int \phi_1 d\nu \int \phi_2 d\nu
\right| 
=
\left|
\int \tilde \phi_1\circ T^n \phi_2 d\nu
\right|
=
\left|\int \tilde \phi_1 \mathcal{L}^n \phi_2 d\nu\right|
\]
which is bounded by $|\tilde \phi_1|_\infty c_1 \varsigma^n \|\phi_2\| + \int \tilde \phi_1 d\nu \int \mathcal{L}^n \mathcal{P}\phi_2 d\nu$ with the second term being $0$. On adjusting the constant $c_1$, we finish the proof.
\end{proof}

\section{Almost sure invariance principle}

\begin{thm}
The invariant measure $\nu$ satisfies the almost sure invariance principle.
\end{thm}

For a function $f\in \mathcal{C}$ (note that this $f\in L^\infty(\nu)$) and for $t$ sufficiently close to 0, define a family of operators $\mathcal{L}_{t,f}, t\in\mathbb{R}$, on $\mathcal{C}$ by
\[
\mathcal{L}_{t,f}(\Phi) = \mathcal{L} \left( e^{it f} \Phi\right).
\] We now recall an abstract theorem by Gou\"ezel \cite{Go2010}, stated in a manner relevant to our setting.

\begin{proposition}
Suppose $\mathcal{L}$ satisfies Theorem \ref{t:LY-perturb} and there exists a constant $C_8>0$ such that $\| \mathcal{L}_{f,t}^n \|_{\mathcal{C}\to\mathcal{C}} \leq C_8$ for all $n\in\mathbb{N}$ and for all $t$ small enough. Then there exists a probability space $(\Gamma)$ and two processes $(A_j)$ and $(B_j)$ on $\Gamma$ such that
\begin{enumerate}
\item the processes $(f\circ T^j)$ and $(A_j)$ have the same distribution
\item the random variables $(B_j)$ are independent and distributed as $\mathcal{N}(0, \sigma^2)$ for an appropriately chosen $\sigma^2$, and
\item almost surely in $\Gamma$
\[
\left|
\sum_{l = 0}^{n-1} A_j - \sum_{l = 0}^{n-1} B_j
\right| = o(n^\gamma)
\] for any $\gamma > 0.25.$
\end{enumerate} 
\end{proposition}

As noted in their paper, since a Brownian motion at integer times coincides with a sum of iid Gaussian random variables, this theorem can be formulated as as an almost sure approximation by a Brownian motion. To establish the almost sure invariance principle for our setup, we only need to check that $\| \mathcal{L}_{f,t}^n \|_{\mathcal{C}\to\mathcal{C}}$ stays bounded for $t$ small enough.

\begin{lemma}
There exists a constant $C_9>0$ such that $\| \mathcal{L}_{f,t}^n \|_{\mathcal{C}\to\mathcal{C}} \leq C_9$ 
for all $n\in\mathbb{N}$ and for all $t$ small enough.
\end{lemma}

\begin{proof}
Since $|\mathcal{L}_{f,t}^n (\Phi)|_\infty \leq |e^{itf}\Phi|_\infty \leq |\Phi|_\infty$, we only need to check if 
$|\mathcal{L}_{f,t}^n (\Phi)|_\beta $ is bounded. By Theorem \ref{t:LY-perturb} 
\[
|\mathcal{L}_{f,t}^n (\Phi)|_\beta  \leq |e^{itf}\Phi|_\beta  (C_E\eta)^n + C_6 |\Phi|_\infty.
\] 
It remains to bound $|e^{itf}\Phi|_\beta $. Indeed, by the triangle inequality:
\begin{eqnarray*}
\left|
e^{itf(x)}\Phi(x) - e^{itf(y)}\Phi(y)
\right| & \leq & 
\left|
e^{itf(x)}\Phi(x) - e^{itf(x)}\Phi(y)
\right| +
\left|
e^{itf(x)}\Phi(y) - e^{itf(y)}\Phi(y)
\right|\\
& \leq& |\Phi(x) - \Phi(y)| + |\Phi|_\infty |e^{itf(x)} - e^{itf(y)}|\\
& \leq & |\Phi|_\beta d(x, y)^\beta + |\Phi|_\infty c_1 d(x, y)^\beta + o(d(x, y)^\beta)
\end{eqnarray*}
where the constant $c_1$ depends on $f$ (as well as $\eta$ and $\beta$). 
The last inequality uses an argument from lemma~\ref{lemma:LYlk}. On dividing by $d(x, y)^\beta$ 
and taking the supremum we obtain (for some $c_2$)
\[
|e^{itf}\Phi|_\beta  \leq |\Phi|_\beta + c_2 |\Phi|_\infty.
\] 
Hence we can put
\[
C_9:= \max \set{\sup_n \set{(|\Phi|_\beta  + c_2 |\Phi|_\infty) (C_E\eta)^n} + C_6, 1}
\] 
which is finite.
 \end{proof}

\bibliographystyle{alpha}
\bibliography{biblio}

\end{document}